\documentclass[11pt]{elsarticle} 
\usepackage{amsmath}
\usepackage{amssymb}
\usepackage{amsthm,mathtools}
\usepackage{lineno}
\usepackage[mathscr]{eucal}
\usepackage{xcolor}
\usepackage{fontenc}
\usepackage{graphicx}
\usepackage{geometry}
\usepackage{lineno}

\theoremstyle{definition}
\newtheorem{definition}{Definition}[section]
\newtheorem{theorem}[definition]{Theorem}
\newtheorem*{theorem*}{Conjecture}
\newtheorem{proposition}[definition]{Proposition}
\newtheorem{lemma}[definition]{Lemma}

\theoremstyle{remark}
\newtheorem{remark}[definition]{Remark}
\newtheorem{example}[definition]{Example}
\newcounter{enumctr}

\newcommand{\R}{\mathbb{R}}
\newcommand{\N}{\mathbb{N}}

\def\du{{\ensuremath{\mathrm{d}}}}

\begin{document}
\begin{frontmatter}
	\title{On the Mittag-Leffler stability of mixed-order fractional homogeneous cooperative delay systems} 
\author[add1]{La Van Thinh}
\author[add3]{Hoang The Tuan}
\address[add1]{Academy of Finance, No. 58, Le Van Hien St., Duc Thang Wrd., Bac Tu Liem Dist., Hanoi 10307, Viet Nam}
\address[add3]{Institute of Mathematics, Vietnam Academy of Science and Technology, 18 Hoang Quoc Viet, Ha Noi 10307, Viet Nam}
\begin{abstract}
In this paper, we study a class of multi-order fractional nonlinear delay systems. Our main contribution is to show the (local or global) Mittag-Leffler stability of systems when some structural assumptions are imposed on the ``vector fields": cooperativeness, homogeneity, and order-preserving on the positive orthant of the phase space. In particular, our method is applicable to the case where the degrees of homogeneity of the non-lag and lag components of the vector field are different. In addition, we also investigate in detail the convergence rate of the solutions to the equilibrium point. Two specific examples are also provided to illustrate the validity of the proposed theoretical result.
\end{abstract}

\end{frontmatter}
\section{Introduction}
A Positive system (a system with the property that a non-negative input will result in a non-negative output) plays an important role in modeling many important problems in real life because many quantities in physics, state variables in chemical reactions and bio-ecological models are naturally constrained to be non-negative. Besides, the delay differential equation is an important subject in the qualitative theory of dynamical systems because many realistic processes and phenomena depend on history. Therefore, there has been a large amount of published literature concerned with the delay positive systems (see, e.g., \cite{benvenuti-farina,Haddad_book,Hien21,Rantzer}).

Consider the simplest linear differential system
\begin{equation}\label{introeq1}
\begin{cases}
     \displaystyle\frac{d}{dt}  x(t)&=Ax(t),\;\forall t\geq 0,\\
     x(0)&=x_0\in\R^d,
     \end{cases}
\end{equation}
 where $A\in\R^{d\times d}$. It is not difficult to check that this system is {\em positive} if and only if $A$ is {\em Metzler}. Moreover,  the positive system \eqref{introeq1} is {\em asymptotically stable} if and only if there is a positive $v\succ 0$ with $Av \prec 0$ (see, e.g., \cite{Farina}).
For the simplest form of delayed linear systems
\begin{equation}\label{introeq2}
\begin{cases}
     \displaystyle\frac{d}{dt}  x(t)&=Ax(t)+Bx(t-r),\;\forall t\geq 0,\\
     x(0)&=x_0\in\R^d,
     \end{cases}
\end{equation}
 where $A,B\in\R^{d\times d}$, $r\geq 0$, in \cite{Haddad}, the authors have shown that it is positive if and only if $A$ is Metzler, $B$ is nonnegative. They have also proven that the positive delay system \eqref{introeq2} is asymptotically stable if and only if there exists a positive $v\succ 0$ satisfying $(A+B)v \prec 0$. Then, a such result has been extended
to cooperative homogeneous systems with the degree $\alpha=1$ by O. Mason and M. Verwoerd \cite{Mason}. Later, in \cite{Bokharaie}, it was further extended to cooperative delay systems with the degree of homogeneity $\alpha>0$ (with respect to a dilation map). Recently, there have been contributions on the convergence rate of solutions of generalized cooperative homogeneous systems with bounded delays to their equilibrium point established by J.G. Dong \cite{Dong} and Q. Xiao et al. \cite{Xiao}.

Fractional calculus is a useful and suitable tool for describing processes or materials' memory and hereditary properties. It is a significant advantage over classical models where such effects are ignored. For interested readers, the latest applications of fractional order differential equations can be found in the survey paper \cite{Sun_18} and updated monographs, see, for example, \cite{Baleanu_1,Baleanu_2, Petras, Tarasov_1,Tarasov_2} and references therein.

For the above reasons, fractional positive delay systems promise to be a useful tool
in describing the dynamic properties of memory-dependent phenomena.

The two biggest challenges in studying fractional-order differential equations are that their solutions are non-local and the fractional-order derivatives have no geometric explanation. These lead to the fact that one cannot apply Lyapunov's classic methods to these equations. The situation becomes especially difficult for non-commensurate systems where a variation of constants formula which is the most essential part of the linearization approach is absent. Recently, H.T. Tuan and L.V. Thinh \cite{Tuan_Thinh_Esaim, Thinh_Tuan_CNSNS} explored that the solutions of non-commensurate positive linear equations have separation properties. They then developed comparative arguments to analyze the solutions of these systems.

With the desire to design a system in which the positivity of the solutions is guaranteed, the time delay dependence is expressed and the influence of the entire past of the process is reflected, inspired by \cite{Dong,Tuan_Thinh_Esaim,Thinh_Tuan_CNSNS}, we are interested in non-commensurate nonlinear delay systems with some structural assumptions imposed on the vector fields so that the order relation on the phase space is preserved. More precise, our main object in the paper is the system:
\begin{equation} \label{Eq main}
	\begin{cases}
		^{\!C}D^{{\hat\alpha}}_{0^+}{w}(t)&={ f}({ w}(t))+\displaystyle\sum_{j=1}^m{ g^{(j)}}({ w}(t-\tau_j(t))),\ \forall t> 0,\\
		{w}(s)&={\varphi}(s), \ \forall s\in[-r,0],
	\end{cases}
\end{equation}
where $\hat\alpha\in (0,1]\times\cdots\times (0,1]$, $m\geq 1$, $r_j>0$, $j=1,\dots,m$, are given nonnegative constants, $\tau_j:[0,\infty)\rightarrow [0,r_j]$, $1\leq j\leq m$, are continuous, $r=\displaystyle\max_{1\leq j\leq m}r_j$ and  ${\varphi}:[-r,0]\rightarrow\R_{\geq 0}^d$ is a continuous initial condition, $f(\cdot)$, ${g^{(j)}}(\cdot)$, $j = 1,\dots, m,$ satisfy following assumption.\\

\noindent{\bf Assumption (H1)}: ${f}(\cdot)$  is {\em cooperative}\footnote{see in Definition \ref{Dfcooperative}} on $\R^d$ and is {\em homogeneous}\footnote{see in Definition \ref{Dfhomogeneous}} of degree $p\geq 1$.\\
{\bf Assumption (H2)}: ${g^{(j)}}(\cdot)$ is {\em order-preserving}\footnote{see in Definition \ref{Dfpreserving}} on $\R^d_{\ge0}$ and is homogeneous of degree $q_j\geq p\geq 1$.\\
{\bf Assumption (S)}: There exists a vector ${v} \succ {0}$ such that ${f}({v})+\displaystyle \sum_{j=1}^m{g^{(j)}}({v})\prec {0}.$ 

Based on some special features of the Mittag-Leffler functions and comparison arguments, our main contribution is to prove the Mittag-Lefler stability of system \eqref{Eq main}. In particular, we explore the rate of convergence of the solutions to its equilibrium point. Additionally, depending on the degree of homogeneity of the functions $f$ and $g^{j}$, $j=1,\dots,m$, a result on local or global attractiveness of the equilibrium point will be derived. This is a continuation of recent published papers, see, e.g., \cite{Mason, Feyzma_2, Dong, Tuan_Trinh_Siam, THL_21, Tuan_Thinh_Esaim, KT_23}. Finally, numerical examples are provided to illustrate the theoretical findings.
\section{Notations and preliminaries}
\subsection{Notations}
Throughout the paper, the following notations are used: $\R,\ \N$ is the set of real
numbers, natural numbers, $\R_{\geq 0}:=\{x\in \R:x\geq 0\},\ \R_+:=\{x\in \R:x> 0\};$ $\R^d$ stands for the $d$-dimensional
Euclidean space, $\R_{\geq 0}^d$ is the subset of $\R^d$ with
nonnegative entries and ${\R_+^d}:=\left\{{x}=(x_1,\dots,x_d)^{\rm T}\in \R^d:x_i >0,\ 1\le i\le d\right\}$. For two vectors ${ w,u}\in \R^d$, we write
\begin{itemize}
	\item ${u}\preceq {w}$ if $u_i\leq w_i$, $1\leq i\leq d.$  
	\item ${u}\prec {w}$ if $u_i< w_i$, $1\leq i\leq d.$
\end{itemize}
Let $r>0$, we denote $B_r({0}):=\{{x}\in \R^d:\|{x}\|\leq r\}$ and  $\partial B_r({0}):=\{{x}\in \R^d:\|{x}\|= r\}$. For a vector valued function  ${ f}: \R^d \longrightarrow \R^d$ which is differentiable at ${ x}\in \R^d$, we set $D{ f}({ x}):=\big(\displaystyle\frac{\partial f_i}{\partial x_j}({x})\big)_{1\leq i,j\leq d}$. Fixing a vector ${ v}\succ 0$, the weighted  norm $\|.\|_{ v}$ is given by	
$$\|{w}\|_{v}:=\max_{1\le i \le d} \frac{|w_i|}{v_i}.$$
A real matrix $A=(a_{ij})_{1\leq i,j\leq d}$ is called as Metzler if its off-diagonal entries $a_{ij},\ \forall i\ne j$, are nonnegative. 

Let $\alpha \in (0,1]$ and $J = [0, T]$, the Riemann-Liouville fractional integral of a function $x :J \rightarrow \mathbb R$ is as 
	$$ 
		I^\alpha_{0^+}x(t) := \frac{1}{\Gamma(\alpha)}\int_{0}^{t}(t-s)^{\alpha -1}x(s) \, \du s,\ \quad t\in J,
	$$
	and the Caputo fractional derivative of the order $\alpha$ is given by
	$$ 
		^C D^\alpha_{0^+}x(t) := \frac{d}{dt}I^{1-\alpha}_{0^+}(x(t) - x(0)), \quad t \in J \setminus \{ 0 \},
	$$
	here $\Gamma(\cdot)$ is the Gamma function, $\displaystyle\frac{d}{d t}$ is the first derivative (see, e.g., 
	\cite[Chapters 2 and 3]{Kai} and \cite{Vainikko_16} for more detail on fractional calculus). For $d\in\N,\ {\hat\alpha}:=(\alpha_1,\dots,\alpha_d)\in (0,1]\times\cdots\times (0,1]$ and a function ${ w} : J \rightarrow \R^d,$ then
$$^{\!C}D^{{\hat\alpha}}_{0^+}{ w}(t):=\left(^{\!C}D^{\alpha_1}_{0^+}w_1(t),\dots,^{\!C}D^{\alpha_d}_{0^+}w_d(t)\right)^{\rm T}.$$
\begin{definition}
Let $\alpha,\beta>0$. The Mittag-Leffler function $E_{\alpha,\beta}(\cdot):\R\rightarrow \R$ is defined by
\[
E_{\alpha,\beta}(x)=\sum_{k=0}^\infty \frac{x^k}{\Gamma(\alpha k+\beta)},\;\forall x\in \R.
\]
In the case $\beta=1$, for simplicity we use convention $E_\alpha(x):=E_{\alpha,1}(x)$ for all $x\in\R$.
\end{definition}
\begin{definition}{(see, e.g., \cite[Definition 2.4]{Feyzmahdavian1})}\label{Dfpreserving} 
Let $k,n\in \mathbb N$ and a closed convex cone  $C\subset \R^k$. A function  $f: \R^k \rightarrow \R^n$  is said to be order-preserving on $C$ if $f({ u})\succeq f({ v})$ for any  ${ u},{ v} \in C$ satisfying $ { u} \succeq { v}$.
\end{definition} 
\begin{definition} {\cite[Definition 2.3]{Feyzmahdavian1}}\label{Dfhomogeneous}
	For any $p \ge 0$, a vector field ${ f}: \R^d \longrightarrow \R^d$ is said to be homogeneous of degree $p$ if for all ${x}\in \R^d$ and for all $\lambda >0$, we have
	$${f}({\lambda}({ x}))=\lambda^p {f}({ x}).$$
\end{definition}
\begin{definition} \cite[Definition 2]{Xiao}\label{Dfcooperative}
	A continuous vector field ${f}: \R^d \longrightarrow \R^d$ which is
	continuously differentiable on $\R^d\backslash \left\{0\right\}$ is said to be cooperative if the
	Jacobian matrix $D{f}({x})$ is Metzler for all ${x}\in \R_{\geq 0}^d\backslash \left\{0\right\}$.
\end{definition}
\subsection{Preliminaries}
We provide here some essential materials for further analysis in the next section.
\begin{lemma}
	(see, e.g., \cite[Lemma 3.2]{NNThang_23}, \cite[Lemma 7]{Shuailei_Zhang}) \label{Lemma-ML} If $\eta>0$ and $\alpha\in(0,1]$, then for all $t\ge0,\ s\ge0$, we have
	\[E_\alpha(-\eta t^\alpha)E_\alpha(-\eta s^\alpha)\le E_\alpha(-\eta(t+s)^\alpha)\]
\end{lemma}
\begin{proposition} \cite[Remark 3.1]{H.L.Smith} \label{Pro_Cooperative}
	Let ${f}: \R^d \longrightarrow \R^d$ be a
	cooperative vector field. For any two vectors ${u,w}\in \R_{\geq 0}^d$ with $u_i=w_i,\; i\in\{1,\cdots,d\}$
	and ${u\succeq w}$, we have $$f_i({ u})\ge f_i({ w}).$$
\end{proposition}
\begin{proposition}\cite[Lemma 2.1]{Mason}\label{glcd1}
	Suppose that ${f}: \R^d \rightarrow \R^d$ is continuous and is
	continuously differentiable on $\R^d\backslash \left\{{0}\right\}$. Moreover, this function is homogeneous of degree $p=1$ (or simply homogeneous). Then, there exists a positive constant $K$ such that $$\|{f}({x})-{f}(y)\|\leq K\|{x}-{y}\|,\ \forall { x,\ y}\in \R^d.$$ 
\end{proposition}
\begin{proposition}{\cite[Proposition 2.8]{Thinh_Tuan_23}}\label{loclc}
	Suppose that ${f}: \R^d \rightarrow \R^d$ is continuous and is continuously differentiable on $\R^d\backslash \left\{{0}\right\}$. In addition, we assume that ${f}$ is homogeneous of degree $p>1$. Then, for any $r>0$, we can find a positive constant $K$ such that $$\|{f}({x})-{f}({y})\|\leq K\|{x}-{y}\|,\ \forall { x,\ y}\in \mathcal{B}_r({0}).$$ 
\end{proposition}
\begin{lemma}\label{compare-FDE}
	Let $w : [0, T] \rightarrow \R$ be continuous and assume that the Caputo derivative
	$^{\!C}D^{\alpha}_{0^+}w(\cdot)$ is also continuous on the interval $[0, T]$ with $\alpha\in (0,1]$. If there exists $t_0>0$ such that $w(t_0)=0$ and $w(t)<0,\ \forall t\in [0,t_0)$, then 
 \begin{itemize}
     \item[(i)] $^{\!C}D^{\alpha}_{0^+}w(t_0)> 0$ for $0<\alpha<1$;
     \item[(ii)] $^{\!C}D^{\alpha}_{0^+}w(t_0)\geq 0$ for $\alpha=1$.
 \end{itemize}
\end{lemma}
\begin{proof}
The conclusion of the case $\textup{(ii)}$ is obvious. The proof of the case $\textup{(i)}$ follows directly from \cite[Theorem 1]{Vainikko_16}.
\end{proof}
\section{Mittag-Leffler stability of homogeneous cooperative delay systems}
This part introduces our main contribution concerning the Mittag-Leffler stability of system \eqref{Eq main}. To do this, we first need results concerning the global existence, boundedness, and positivity of the solutions.

\subsection{Boundedness and positivity of solutions}
Consider the multi-order fractional homogeneous cooperative systems with bounded delays \eqref{Eq main} as below.
\begin{equation*} 
	\begin{cases}
		^{\!C}D^{{\hat\alpha}}_{0^+}{w}(t)&={f}({w}(t))+\displaystyle\sum_{j=1}^m{g}^{(j)}({w}(t-\tau_j(t))),\ \forall t> 0,\\
		{w}(s)&=\varphi(s),\ \forall s\in[-r,0].
	\end{cases}
\end{equation*}
Under the assumptions {\textup{\bf (H1)}} and {\textup{\bf (H2)}}, following from Proposition \ref{glcd1}, Proposition \ref{loclc} and the arguments as in the proof of \cite[Theorem 2.2]{Tuan_Trinh_Siam}, for each $\varphi\in C([-r,0];\R^d)$, system \eqref{Eq main} has a unique solution $\Phi(\cdot,\varphi)$ on the maximal interval of existence $[0, T_{max}(\varphi))$.	

Our aim in this subsection is to show the global existence, boundedness, and positivity of the solutions.
\begin{proposition} \label{boundsol}
The following assertions are true.
\begin{itemize}
	\item [(i)] Assume that the conditions {\textup{\bf (H1)}}, {\textup{\bf (H2)}} and {\textup{\bf (S)}} are satisfied. Moreover, the assumption {\textup{\bf (H2)}} is valid for some $q_{j}>p$.  Let $v\succ 0$ is a vector as in {\textup{\bf (S)}}. Then, for any $\varphi\in C([-r,0];\R^d_{\geq 0})$, $\varphi(0)\succ 0$ and $\|\varphi\|_v<1$, the solution $\Phi(\cdot,\varphi)$ of \eqref{Eq main} exists globally on $[0,\infty)$ and
	\[\|\Phi(t,\varphi)\|_{v}\leq \|\varphi\|_{v},\;\forall t\geq 0.\]
	\item [(ii)] Let the conditions {\textup{\bf (H1)}} and {\textup{\bf (S)}} be true and Assumption {\textup{\bf (H2)}} is satisfied for $q_j=p$, $j=1,\dots,m$. Take $v\succ 0$ as in {\textup{\bf (S)}}, then for any $\varphi\in C([-r,0];\R^d_{\geq 0})$, $\varphi(0)\succ 0$, the solution $\Phi(\cdot,\varphi)$ of \eqref{Eq main} exists globally on $[0,\infty)$ and
	\[\|\Phi(t,\varphi)\|_{ v}\leq \|\varphi\|_{ v},\;\forall t\geq 0.\]
\end{itemize}
\end{proposition}
\begin{proof}
\textbf{Case 1:} There exists $q_{j_0}>p$ for some $j_0=1,\dots,m.$ The approach in the proof of this case is similar to that in \cite[Proposition 3.1]{Thinh_Tuan_23}. By virtue of Proposition \ref{glcd1} and Proposition \ref{loclc}, the vector valued functions $f$, $g$ are Lipschitz continuous on $B_r({0})$ for every $r>1.$ Let $v\succ 0$ be a vector as in the asumption {\textup{\bf (S)}}, $\varphi\in C([-r,0];\R^d_{\geq 0})$, $\varphi(0)\succ 0$ and $\|\varphi\|_v<1$. From \cite[Theorem 2.2]{Tuan_Trinh_Siam}, system \eqref{Eq main} has the unique solution $\Phi(\cdot,\varphi)$ on the maximal interval $[0,T_{\max}(\varphi)).$ 
Take ${\epsilon} >0$ be arbitrary such that $\|\varphi\|_v+\varepsilon<1$. For each $i=1,\dots,d$, we define $$y_i(t):=\frac{{\Phi}_i(t,\varphi)}{v_i} - \|\varphi\|_{v}-{\epsilon},\ \forall t\in [0,T_{\max}(\varphi)).$$
From the fact that $$y_i(0)=\frac{\varphi_i(0)}{v_i} - \|\varphi\|_{v}-{\epsilon}<0,\ \forall i=\overline{1,d},$$
due to the continuity of $y_j(\cdot)$, $y_i(t)$ is still positive when $t$ is close enough to $0$. Thus, if there is a $t\in (0,T_{\max}(\varphi))$ and an index $i$ with  $y_i(t)=0$, by choosing 
	$ t_*:=\inf\{t>0:\exists i= \overline{1,d}\ \text{such that}\ y_{i}(t)=0\},$ 
	then $t_*>0$ and there exists an index $ i^*$ which verify
	\begin{align} \label{t_*^}
		y_{i^*}(t_*)&=0\ \text{and}\ y_{i}(t_*)\le0,\ \forall i\ne i^*, \\
		y_{i^*}(t)&<0,\ \forall t\in[0,t_*),\;i=1,\dots,d. \notag
	\end{align}
Combining \eqref{t_*^} and Lemma \ref{compare-FDE}, it leads to \begin{equation}\label{contrabd}
^{C\!}D^{\alpha_{i^*}}_{0^+}y_{i^*}(t_*)\geq 0.\end{equation}
Furthermore, it is derived from \eqref{t_*^} that
\begin{align} 
	{\Phi}_{i^*}(t_*,\varphi)&=(\|\varphi\|_v+{\epsilon})v_{i^*},\label{estbd1}\\
	{\Phi}_{i}(t,\varphi)&\le(\|\varphi\|_v+{\epsilon})v_{i},\ \forall i=1,\dots,d,\ \forall t\in[0,t_*].\label{estbd2}
\end{align}
Using \eqref{estbd1}, \eqref{estbd2} and Proposition \ref{Pro_Cooperative}, then
\[f_{i^*}(\Phi(t_*,\varphi))\le f_{i^*}\left((\|\varphi\|_{v}+{\epsilon}){v}\right)=\left(\|\varphi\|_{v}+{\epsilon}\right)^p f_{i^*}({v}).\]
On the other hand, from \eqref{estbd2} and the assumption {\textup{\bf (H2)}}, the following estimates hold. 
\begin{itemize}
	\item If $t_*-\tau_j(t_*)\in[0,t_*]$, then
	\[g_{i^*}^{(j)}(\Phi(t_*-\tau_j(t_*),\varphi))\le g_{i^*}^{(j)}\left((\|\varphi\|_{v}+{\epsilon}){ v}\right)=\left(\|\varphi\|_{v}+{\epsilon}\right)^{q_j}g_{i^*}^{(j)}({ v}).\]
	\item If $t_*-\tau_j(t_*)\in[-r,0]$, then 
	\[g_{i^*}^{(j)}(\Phi(t_*-\tau_j(t_*),\varphi))=g_{i^*}^{(j)}(\varphi(t_*-\tau_j(t_*)))\le g_{i^*}^{(j)}\left((\|\varphi\|_{v}+{\epsilon}){v}\right)=\left(\|\varphi\|_{v}+{\epsilon}\right)^{q_j}g_{i^*}({v}).\]
\end{itemize}
Thus, by the observations above,
\[g_{i^*}^{(j)}(\Phi(t_*-\tau_j(t_*),\varphi))\le\left(\|\varphi\|_{v}+{\epsilon}\right)^{q_j}g_{i^*}^{(j)}({ v}).\]
Finally, with the help of the condition {\textup{\bf (S)}}, we see that
	\begin{align*}
		^{C\!}D^{{\alpha_{i^*}}}_{0^+}y_{i^*}(t_*)&=\frac{^{C\!}D^{\alpha_{i^*}}_{0^+}\Phi_{i^*}(t_*,\varphi)}{v_{i^*}} \\
		&=\frac{1}{v_{i^*}}f_{i^*}(\Phi(t_*,\varphi)) + \frac{1}{v_{i^*}}\sum_{j=1}^mg_{i^*}^{(j)}(\Phi(t_*-\tau_j(t_*),\varphi))\\
		&\le\frac{1}{v_{i^*}}\left(\|\varphi\|_{v}+{\epsilon}\right)^pf_{i^*}({v})+\frac{1}{v_{i^*}}\sum_{j=1}^m\left(\|\varphi\|_{v}+{\epsilon}\right)^{q_j} g_{i^*}^{(j)}({v})\\
		&\le\frac{1}{v_{i^*}}\left(\|\varphi\|_{v}+{\epsilon}\right)^pf_{i^*}({v})+\frac{1}{v_{i^*}}\left(\|\varphi\|_{v}+{\epsilon}\right)^p\sum_{j=1}^m g_{i^*}^{(j)}({v})\\
		&=\frac{1}{v_{i^*}}\left(\|\varphi\|_{v}+{\epsilon}\right)^p\left[f_{i^*}({v})+\sum_{j=1}^mg_{i^*}^{(j)}({v})\right]\\
  &< 0,
	\end{align*}
a contradiction with \eqref{contrabd}. This implies that $y_i(t)<0$ all $t\in[0,T_{\max}(\varphi))$ and for all $i=1,\dots,d$. Hence,
$$\frac{{\Phi}_i(t,\varphi)}{v_i} < \|\varphi\|_{v} +{\epsilon},\ \forall t\in[0,T_{\max}(\varphi)),\ i=1,\dots,d.$$
Let ${\epsilon}\to 0$, we obtain
$$\frac{{\Phi}_i(t,\varphi)}{v_i} \le \|\varphi\|_{v},\ \forall t\in[0,T_{\max}(\varphi)),\ i=1,\dots,d$$ or
\begin{equation}
	\label{est-Tmax} \|\Phi(t,\varphi)\|_{ v}\le \|\varphi\|_{v},\ \forall t\in[0,T_{\max}(\varphi)).
\end{equation}
However, in light of \eqref{est-Tmax} and the definition of the maximal interval of existence, it must be true that $T_{\max}(\varphi)=\infty$ because otherwise the solution $\Phi(\cdot,\varphi)$ can be extended over a larger interval.

\textbf{Case 2:} $q_j=p>1$ for all $j=1,\dots,m$. The arguments in the proof of {\bf Case 1} are still valid without the additional condition $\|\varphi\|_v<1$.

\textbf{Case 3:} $q_j=p=1$ for all $j=1,\dots,m$. By Proposition \ref{glcd1}, the functions $f$, $g$ are global Lipschitz continuous on $\R^d$. Therefore, from \cite[Theorem 2.2]{Tuan_Trinh_Siam}, for any $\varphi\succeq{0}$ on $[-r,0]$, $\varphi(0)\succ 0$, system \eqref{Eq main} has the unique global nonnegative solution on $[0,\infty)$. Now, repeating the arguments in the proof of {\bf Case 1}, it is easy to see that
\[\|\Phi(t,\varphi)\|_{v}\le \|\varphi\|_{v},\ \forall t\in[0,\infty).\]
The proof is complete.
\end{proof}
\begin{proposition}\label{Sys positive}
Consider system \eqref{Eq main}. Suppose that the assumption {\bf{\textup{(H1)}}}, {\bf{\textup{(H2)}}} and {\bf{\textup{(S)}}} are satisfied. 
\begin{itemize}
	\item [(i)] In addition, assume that {\bf{\textup{(H2)}}} is verified for some $q_j>p$. Let $v\succ 0$ as in {\bf{\textup{(S)}}}. Then, for any $\varphi\in C([-r,0];\R^d_{\geq 0})$ with $\|\varphi\|_v<1$, the solution $\Phi(\cdot,\varphi)$ exists globally and is non-negative on $[0,\infty)$. Moreover, $\|\Phi(t,\varphi)\|_v\leq \|\varphi\|_v$ for all $t\geq 0$.
	\item [(ii)] Assume that {\bf{\textup{(H2)}}} is true with $q_j=p$, $j=1,\dots,m$. Then, for any $\varphi\in C([-r,0];\R^d_{\geq 0})$, the solution $\Phi(\cdot,\varphi)$ exists globally on $[0,\infty)$ and $\Phi(t,\varphi)\succeq 0$ for all $t\geq 0$. Furthermore, we also obtain the estimate $\|\Phi(t,\varphi)\|_v\leq \|\varphi\|_v$ for all $t\geq 0$.
	\end{itemize}
\end{proposition}
\begin{proof} {\bf Case 1}: There exists some $q_j>p$. Take and fix the initial condition $\varphi\succeq {0}$ with $\|\varphi\|_v<1$. Choose $k$ large enough such that $\|\varphi\|_v+\frac{1}{k}<1$. Let $\Phi^k(\cdot,\varphi^k)$ be the unique solution of the initial value problem
\begin{equation} \label{Eq main_n}
	\begin{cases}
		^{\!C}D^{{\hat\alpha}}_{0^+}{x}(t)&={f}({x}(t))+\displaystyle\sum_{j=1}^m{g}^{(j)}({x}(t-\tau_j(t)))+\frac{\textbf{e}}{k},\ \forall t>0,\\
		{x}(s)&=\varphi^k(s),\ \forall s\in[-r,0],
	\end{cases}
\end{equation}
where $\varphi^k(s)=\varphi(s)+\displaystyle\frac{1}{k}\textbf{e},\ s\in[-r,0]$ and  $\textbf{e}:=(1,\dots,1)^{\rm T} \in \R^d$. It follows from 
Proposition \ref{boundsol} that $\Phi^k(t,\varphi^k)\succ 0$ and $\|\Phi^k(t,\varphi^k)\|_v\leq \|\varphi^k\|_v$ for all $t\geq 0$. Let $k,n\in \mathbb N$, $n>k$ and put $\eta(t):=\Phi^k(t,\varphi^k)-\Phi^n(t,\varphi^n),\ \forall t\in [0, \infty).$ Suppose that there exists a $t>0$ and an index $i=1,\dots,d$ with $\eta_{i}(t)=0$. Take
	\begin{align*}
		t_0:=\inf \{t>0:\exists i= \overline{1,d}\ \text{such that}\ \eta_{i}(t)=0\}.
	\end{align*}  
This implies $t_0>0$. Furthermore, there is an index $ i_0$ so that  
	\begin{align} \label{t_**}
	\eta_{i_0}(t_0)&=0,\quad \eta_{i}(t_0)\ge0,\ i\ne i_0,\\
	\quad \eta_{i_0}(t)&>0,\ \forall t\in[0,t_0). \notag
\end{align}
Since $\eta_{i_0}(t_0)=0$ and $\eta_{i_0}(t)>0,\ \forall t\in [0,t_0)$, by Lemma \ref{compare-FDE}, it deduces that  $$^{\!C}D^{\alpha_{i_0}}_{0^+}\eta_{i_0}(t_0)\leq  0.$$ 
On the other hand, from \eqref{t_**}, we see
	\begin{align*}
		\Phi^k_{i_0}(t_0,\varphi^k)&=\Phi^n_{i_0}(t_0,\varphi^n),\\
		\Phi^n_{i}(t,\varphi^n)&\le\Phi^k_{i}(t,\varphi^k),\ \forall i\ne i_0,\ \forall t\in[0,t_0],
	\end{align*}
	which together with Proposition \ref{Pro_Cooperative} and the fact that $f$ is cooperative implies $$f_{i_0}(\Phi^n(t_0,\varphi^n))\le f_{i_0}(\Phi^k(t_0,\varphi^k)).$$
	With the help of the assumption that $g^{(j)}$ is order-preserving, we obtain
	\begin{itemize}
		\item if $t_0-\tau_j(t_0)\ge0$, then $$ g_{i_0}^{(j)}(\Phi^n(t_0-\tau_j(t_0)),\varphi^n)\leq g_{i_0}^{(j)}(\Phi^k(t_0-\tau_j(t_0)),\varphi^k);$$ 
		\item  if $t_0-\tau_j(t_0)<0$, then $$ g_{i_0}^{(j)}(\Phi^n(t_0-\tau_j(t_0)),\varphi^n)=\varphi^n(t_0-\tau_j(t_0))\leq g_{i_0}^{(j)}(\Phi^k(t_0-\tau_j(t_0)),\varphi^k)=\varphi^k(t_0-\tau_j(t_0));$$
	\end{itemize}
	and thus
	$$ g^{(j)}_{i_0}(\Phi^n(t_0-\tau_j(t_0)),\varphi^n)\le g_{i_0}^{(j)}(\Phi^k(t_0-\tau_j(t_0)),\varphi^k).$$
	These lead to that 
	\begin{align*}
		^{\!C}D^{\alpha_{i_0}}_{0^+}\eta_{i_0}(t_0)&=	^{\!C}D^{\alpha_{i_0}}_{0^+}\Phi^k_{i_0}(t_0,\varphi^k)-	^{\!C}D^{\alpha_{i_0}}_{0^+}\Phi^n_{i_0}(t_0,\varphi^n) \\
		&=[f_{i_0}(\Phi^k(t_0),\varphi^k)+\sum_{j=1}^mg_{i_0}^{(j)}(\Phi^k(t_0-\tau_j(t_0)),\varphi^k)+\frac{1}{k}]\\
  &\hspace{1cm}-[f_{i_0}(\Phi^n(t_0),\varphi^n)+\sum_{j=1}^mg_{i_0}^{(j)}(\Phi^n(t_0-\tau_j(t_0)),\varphi^n)+\frac{1}{n}]\\
		&=f_{i_0}(\Phi^k(t_0),\varphi^k)-f_{i_0}(\Phi^n(t_0),\varphi^n)+\frac{1}{k}-\frac{1}{n}\\
  &\hspace{1cm}+\sum_{j=1}^m[g_{i_0}^{(j)}(\Phi^k(t_0-\tau_j(t_0)),\varphi^k)-g_{i_0}^{(j)}(\Phi^n(t_0-\tau_j(t_0)),\varphi^n)]\\
		&> 0,
	\end{align*}
	a contradiction.
Hence, the sequence $\left\{\Phi^k(\cdot,\varphi^k)\right\}$ (for $k$ large enough) is strictly decreasing on $[0,\infty)$. For each $t\ge0$, the limit $\displaystyle\lim_{k\to\infty}\Phi^k(t,\varphi^k)$ exists. Define
$$\Psi^*(t):=\displaystyle\lim_{k\to\infty}\Phi^k(t,\varphi^k).$$
Using the arguments as in \cite[Theorem 4.2]{Tuan_Thinh_Esaim}, then the sequence $\left\{\Phi^k(\cdot,\varphi^k)\right\}$
converges uniformly to $\Psi^*(\cdot)$ on $[0,T]$ for any $T>0$ and it is obvious to see that $\Psi^*(\cdot)$ is also continuous and nonnegative on this interval. Moreover, $\|\Psi^*(t)\|_v\leq \|\varphi\|_v$ for all $t\geq 0$. Now, based on the integral form of the solution
\begin{align*}
	\Phi_i^k(t,\varphi^k)&=\varphi_i(0)+\frac{1}{\Gamma(\alpha_i)}\int_0^t(t-s)^{\alpha_i-1}\big[f_i(\Phi^k(s,\varphi^k))+\sum_{j=1}^mg_i^{(j)}(\Phi^k(s-\tau_j(s),\varphi^k))\big]ds\\
 &\hspace{1cm}+\frac{1}{k}+\frac{t^{\alpha_i}}{k\Gamma(\alpha_i+1)}
\end{align*} 
for all $t\geq 0$ and letting  $k \rightarrow \infty$, then
\begin{align*}
	\Psi_i^*(t)=\varphi_i(0)+\frac{1}{\Gamma(\alpha_i)}\int_0^t(t-s)^{\alpha_i-1}\big[f_i(\Psi^*(s))+\sum_{j=1}^mg_i^{(j)}(\Psi^*(s-\tau_j(s)))\big]ds
\end{align*} 
for all $t\geq 0,\;i=1,\dots,d$.
This together with the fact system \eqref{Eq main} has a unique solution which exists globally on $[0,\infty)$ shows that  $\Psi^*(t) = \Phi(t,\varphi),\ \forall t\geq 0$. In particular, $\Phi(\cdot,\varphi)$ is non-negative and $\|\Phi(t,\varphi)\|_v\leq \|\varphi\|_v$ for all $t\geq 0$.

{\bf Case 2}: $q_j=p$ for all $j=1,\dots,m$. Using the same arguments mentioned above and note that the condition $\|\varphi\|_v<1$ is not necessary.
\end{proof}
\subsection{Mittag-Leffler stability of the systems}
This subsection presents our main contribution concerning the asymptotic properties and convergence rate of the solution of the system \eqref{Eq main} to the origin. Before stating the main result, we introduce the definitions of Mittag-Leffler stability that were previously established in \cite{Tuan_Trinh_Siam,Tuan_Cong_JMAA}.
\begin{definition} \label{GMLS-def}
	The trivial solution of \eqref{Eq main} is called globally Mittag-Leffler stable if there exists a positive parameter $\beta$ such that for any $\varphi\in C([-r,0];\R^d)$, the solution $\Phi(\cdot,\varphi)$ exists on the interval
	$[0,\infty)$ and satisfies
	\[\|\Phi(t,\varphi)\|\le \nu E_{\beta}(-ct^\beta),\ \forall t\ge0,\]
$c, \nu>0$ are parameters depending on $\varphi$, $\alpha$, $f$ and $g^{(j)}$, $j=1,\dots,m$.	
\end{definition}
\begin{definition} \label{LMLS-def}
	The trivial solution of \eqref{Eq main} is called locally Mittag-Leffler stable if for any $\varphi\in C([-r,0];\R^d)$ with $\|\varphi\|$ small enough, the solution $\Phi(\cdot,\varphi)$ exists on the interval
	$[0,\infty)$ and satisfies
	\[\|\Phi(t,\varphi)\|\le \nu E_{\beta}(-ct^\beta),\ \forall t\ge0,\]
	where $\beta$ is positive constant independent of the initial condition $\varphi$, and $c, m>0$ are parameters depending on $\varphi$, $\alpha$, $f$ and $g^{(j)}$, $j=1,\dots,m$. 
\end{definition}
	\begin{theorem} \label{LMLS}
	Suppose that assumptions {\bf{\textup{(H1)}}}, {\bf{\textup{(H2)}}} and  {\bf{\textup{(S)}}} are true. Then, the trivial solution of \eqref{Eq main} is locally Mittag-Leffler stable. 
	\end{theorem}
	\begin{proof}
Let $v\succ 0$ be a vector satisfying the assumption {\textup{\bf(S)}}. Define $\beta:=\displaystyle\min_{1\leq i\leq d}\alpha_i/p$. Consider the initial condition 	$\varphi\in C([-r,0];\R^d_{\geq 0})$ with $\|\varphi\|_v<1$. Without loss of generality, we will only focus on the case $\varphi\neq 0$. By virtue of the assumption {\bf{\textup{(H2)}}}, for all $i=1,\dots,d$, we can find a constant $c\in (0,1)$ verifying the following inequality
\begin{equation}\label{addeq1}
\frac{f_i({v})}{v_i}+\sum_{j=1}^m\frac{1}{E_\beta(-c r^\beta)^{q_j}}\frac{g_i^{(j)}({v})}{v_i}+\Big(\displaystyle\frac{\|\varphi\|_{v}}{E_\beta(-c)}\Big)^{1-p}c\sup_{t\ge1}I_i(t)\leq 0,
\end{equation}
where $I_i(t):=\frac{t^{\beta-\alpha_i}E_{\beta,\beta+1-\alpha_i}(-c t^\beta)}{E_\beta(-c t^\beta)}$. Take a constant $\epsilon>0$ with $\|\varphi\|_{v}+\epsilon\le 1$ and denote $\nu_\varepsilon:=\displaystyle\frac{\|\varphi\|_{v}+\epsilon}{E_\beta(-c)}$.
We will prove that
$$0\le\Phi_i(t,\varphi)< \nu_\varepsilon E_{\beta}(-ct^{\beta}), \; \forall t\ge0,\;i=1,\dots, d.$$
To do this, we first set \[z_i(t):=\frac{\Phi_i(t,\varphi)}{v_i}-\nu_\varepsilon E_{\beta}(-ct^{\beta}),\ t\ge0,\ i=1,\dots,d.\]
From Proposition \ref{boundsol} and the proof of Proposition \ref{Sys positive}, we see that $\|\Phi(t,\varphi)\|_{v}\le \|\varphi\|_{v},\ \forall t\ge 0$. Hence, $z_i(t)<0$ for all $t\in[0,1]$, and $i=1,\dots,d.$ Thus, if the assertion that ${z}(t)\prec {0},\ \forall t\ge0$ is false, there exist $t_*>1$ and $i_*$ so that
\begin{equation}
	\label{t_*-Theorem-main} \begin{cases}
		z_{i_*}(t_*)&=0\ \text{and}\ z_i(t_*)\le0,\ \forall i\ne i_*\\
		z_{i_*}(t)&<0,\ \forall t\in[0,t_*).
		\end{cases}
\end{equation}
This means that
\begin{equation}\label{rv_est1} 
\begin{cases}
		\Phi_{i_*}(t_*,\varphi)&=\nu_\varepsilon E_{\beta}(-ct^{\beta})v_{i_*}\ \text{and}\ \Phi_i(t_*,\varphi)\le m_\varepsilon E_{\beta}(-ct^{\beta})v_{i},\ \forall i\ne i_*\\
		\Phi_{i_*}(t_*,\varphi)&<\nu_\varepsilon E_{\beta}(-ct^{\beta})v_{i_*},\ \forall t\in[0,t_*).
	\end{cases}
\end{equation}
Due to the fact that $f$ is cooperative (Proposition \ref{Pro_Cooperative}) and that $f$ is homogeneous of degree $p$, it deduces from \eqref{rv_est1} that
\[f_{i_*}(\Phi(t_*,\varphi))\le f_{i_*}(\nu_\varepsilon E_{\beta}(-ct^{\beta}){v})=\left(\nu_\varepsilon E_{\beta}(-ct_*^{\beta})\right)^pf_{i_*}({v}).\]
Notice that, by \eqref{rv_est1}, the assumption {\bf(H1)} ($g^j(\cdot)$, $j=1,\cdots,m$, is homogeneous of degree $q_j$ and is order-preserving on $\R^d_{\geq 0}$), Lemma \ref{Lemma-ML}, and the fact that $E_\beta(-ct^\beta)$ is strictly deceasing on $[0,\infty)$, we obtain the estimates below.
\begin{itemize}
	\item If $t_*-\tau_j(t_*)\in[0,t_*]$, then
	\begin{align}
		g_{i_*}^{(j)}(\Phi(t_*-\tau_j(t_*),\varphi))&\le g_{i_*}^{(j)}(\nu_\varepsilon E_{\beta}(-c(t_*-\tau_j(t_*))^{\beta}){ v})\notag\\
		&=(\nu_\varepsilon E_{\beta}(-c(t_*-\tau_j(t_*))^{\beta}))^{q_j}g_{i_*}^{(j)}({ v})\notag\\
		&\le \left(\frac{\nu_\varepsilon E_{\beta}(-ct_*^\beta)}{E_{\beta}(-c\tau_j(t_*)^\beta)}\right)^{q_j}g_{i_*}^{(j)}({ v})\notag\\
		&\le \frac{\left(\nu_\varepsilon E_{\beta}(-ct_*^\beta)\right)^p}{E_{\beta}(-cr^\beta)^{q_j}}g_{i_*}^{(j)}({ v}).\label{addeq2}		
	\end{align}
	\item If $t_*-\tau_j(t_*)\in[-r,0]$, then
	\begin{align}\label{addeq3}
		g_{i_*}^{(j)}(\Phi(t_*-\tau_j(t_*),\varphi))\le g_{i_*}^{(j)}(\nu_\varepsilon {v})=\nu_\varepsilon ^{q_j}g_{i_*}^{(j)}({v})
		&\le \frac{\left(\nu_\varepsilon E_{\beta}(-ct_*^\beta)\right)^{q_j}}{E_{\beta}(-cr^\beta)^{q_j}}g_{i_*}^{(j)}({v})\notag \\
		&\le\frac{\left(\nu_\varepsilon E_{\beta}(-ct_*^\beta)\right)^p}{E_{\beta}(-cr^\beta)^{q_j}}g_{i_*}^{(j)}({v}).		
	\end{align}
\end{itemize}
Combining \eqref{addeq2}--\eqref{addeq3}, this deduces
\[g_{i_*}^{(j)}(\Phi(t_*-\tau_j(t_*),\varphi))\le\frac{\left(\nu_\varepsilon 
	E_{\beta}(-ct_*^\beta)\right)^p}{E_{\beta}(-cr^\beta)^{q_j}}g_{i_*}^{(j)}({ v}).\]
Moreover, using a direct computation, it is obvious to see that 
\[^{\!C}D^{{\alpha_{i_*}}}_{0^+}\left(E_{\beta}(-ct_*^\beta)\right)=-ct_*^{\beta-\alpha_{i_*}}E_{\beta,1+\beta-\alpha_{i_*}}(-ct_*^\beta).\]
Thus,
\begin{align*}
	^{\!C}D^{{\alpha_{i_*}}}_{0^+}&z_{i_*}(t_*)=\frac{1}{v_{i_*}}{ ^{\!C}D^{{\alpha_{i_*}}}_{0^+}}\Phi_{i_*}(t_*,\varphi)-^{\!C}D^{{\alpha_{i_*}}}_{0^+}\left(\nu_\varepsilon E_{\beta}(-ct_*^\beta)\right)\\
	&=\frac{1}{v_{i_*}}\left(f_{i_*}(\Phi(t_*,\varphi))+\sum_{j=1}^mg_{i_*}^{(j)}(\Phi(t_*-\tau_j(t_*),\varphi))\right)+\nu_\varepsilon ct_*^{\beta-\alpha_{i_*}}E_{\beta,1+\beta-\alpha_{i_*}}(-ct_*^\beta)\\
	&\le\frac{1}{v_{i_*}}\left[\left(\nu_\varepsilon E_{\beta}(-ct_*^{\beta})\right)^pf_{i_*}({v})+\sum_{j=1}^m\frac{E_{\beta}(-ct_*^\beta)^p}{E_{\beta}(-cr^\beta)^{q_j}}g_{i_*}^{(j)}({v})\right]+ \nu_\varepsilon ct_*^{\beta-\alpha}E_{\beta,1+\beta-\alpha_{i_*}}(-ct_*^\beta)\\
	&=\left(\nu_\varepsilon E_{\beta}(-ct_*^{\beta})\right)^p\left[\frac{f_{i_*}({v})}{v_{i_*}}+\sum_{j=1}^m\frac{g_{i_*}^{(j)}({v})}{E_{\beta}(-cr^\beta)^{q_j}v_{i_*}}+\nu_\varepsilon ^{1-p}c\frac{t_*^{\beta-\alpha_{i_*}}E_{\beta,1+\beta-\alpha_{i_*}}(-ct_*^\beta)}{E_{\beta}(-ct_*^{\beta})^p}\right]\\
	&\le \left(\nu_\varepsilon E_{\beta}(-ct_*^{\beta})\right)^p\left[\frac{f_{i_*}({v})}{v_{i_*}}+\sum_{j=1}^m\frac{g_{i_*}^{(j)}({v})}{E_{\beta}(-cr^\beta)^{q_j}v_{i_*}}+\nu_\varepsilon ^{1-p}c\sup_{t\ge1}\frac{t_*^{\beta-\alpha_{i_*}}E_{\beta,1+\beta-\alpha}(-ct_*^\beta)}{E_{\beta}(-ct_*^{\beta})^p}\right]\\&< 0,
\end{align*}
where the last inequality is derived from \eqref{addeq1}.
However, from \eqref{t_*-Theorem-main} and Lemma \ref{compare-FDE}, this implies that
$^{\!C}D^{{\alpha_{i_*}}}_{0^+}z_{i_*}(t_*)\geq 0$, a contracdiction. From this fact, we obtain that ${z}(t)\prec {0},\ \forall t\ge0$, and thus
\[\frac{\Phi_i(t,\varphi)}{v_i}<\nu_\varepsilon E_{\beta}(-ct^{\beta}),\ \forall t\ge0,\ \forall i=1,\dots,d.\] 
Let $\varepsilon\to 0$, then
\[
\frac{\Phi_i(t,\varphi)}{v_i}\leq \nu E_{\beta}(-ct^{\beta}),\ \forall t\ge0,\ \forall i=1,\dots,d,
\]
here $\nu:=\displaystyle\frac{\|\varphi\|_v}{E_\beta(-c)}$. The proof is complete. 
\end{proof}
\begin{remark}
  Theorem \ref{LMLS} provides a criterion to test the stability and the convergence rate of non-trivial solutions of multi-order fractional cooperative delay systems to the origin. Depending on the situation when the degrees of homogeneity of vector fields are equal or different, we get the global or local stability. To our knowledge, this result has not previously appeared in the literature.  
\end{remark}
\begin{remark}
By \cite[Lemma 4.25, p. 86]{Gorenflo_Book}, the functions $E_{\beta,\beta+1-\alpha_i}(-ct^\beta)$, $i=1,\dots,d$, $E_\beta(-ct^\beta)$  are strictly decreasing on $[0,\infty)$. Hence, for $t\geq 1$, we see that
\[
0\leq \frac{t^{\beta-\alpha_i}E_{\beta,\beta+1-\alpha_i}(-ct^\beta)}{E_\beta(-ct^\beta)}\leq \frac{E_{\beta,\beta+1-\alpha_i}(-ct^\beta)}{E_\beta(-ct^\beta)},\;\forall i=1,\dots,d.
\]
Furthermore, from \cite[Estimate (4.7.5), p. 75]{Gorenflo_Book}, we have
\[
\lim_{t\to\infty}\frac{E_{\beta,\beta+1-\alpha_i}(-ct^\beta)}{E_\beta(-ct^\beta)}=\frac{\Gamma(1-\beta)}{\Gamma(1-\alpha_i)},\;i=1,\dots,d.
\]
Thus, the observations above lead to that $\displaystyle\sup_{t\geq 1}\frac{t^{\beta-\alpha_i}E_{\beta,\beta+1-\alpha_i}(-ct^\beta)}{E_\beta(-ct^\beta)}$ is finite.
\end{remark}
\begin{remark}\label{rmadd2}
If the assumption {\bf (H2)} is true for $q_j=p$ for all $j=1,\dots,m$, then the proof of Theorem \ref{LMLS} holds without requiring the initial condition $\varphi(\cdot)$ to be small. Hence, in this case, the trivial solution is globally Mittag-Leffler stable.	
\end{remark}
\begin{remark}
	Consider system \eqref{Eq main} when $\alpha_1=\dots=\alpha_d=\alpha_0\in (0,1)$. Then, from the proof of Theorem \ref{LMLS} and Remark \ref{rmadd2}, the trivial solution is locally Mittag-Leffler stable or globally Mittag-Leffler stable and the optimal convergence rate of the solutions to the origin as $t^{-\alpha_0/p}$.
	\end{remark}
\begin{remark}
	Suppose that the assumption {\bf (H2)} is true for $p=q_j=1$ for all $j=1,\dots,m$ and  $\alpha_1=\dots=\alpha_d=\alpha_0\in (0,1)$. Then, the trivial solution of system \eqref{Eq main} is globally Mittag-Leffler stable. In particular, let $v\succ 0$ satisfying the assumption {\bf (S)}, based on the arguments as in the proof of Theorem \ref{LMLS}, for any initial condition $\varphi\in C([-r,0];\R^d_{\geq 0})$, the following optimal estimates hold
	\[
	\Phi_i(t,\varphi)\le \|\varphi\|_v E_{\alpha}(-\eta t^{\alpha_0}),\ \forall t\ge0,\;i=1,\dots,d,
	\]
	where $\eta>0$ is some constant such that
	\[
	f_i(v)+\sum_{j=1}^mg_i^{(j)}(v)+\eta\leq 0,\quad \forall i=1,\dots,d.
	\]
\end{remark}
\begin{remark}
The approach as in the proof of Theorem \ref{LMLS} still holds when $\beta=1$ or $\alpha_i=1$ for all $i=1,\dots,d$.   
\end{remark}


\section{Numerical examples and discussion} 
We present two numerical examples to illustrate the proposed theoretical result.
\begin{example}
	Consider the system
	\begin{equation} \label{Eqmain-2}
		\begin{cases}
		^{\!C}D^{\hat{\alpha}}_{0^+}w(t)&=f(w(t))+g(w(t-\tau(t)),\ \forall t> 0,\\
			w(s)&=\varphi(s),\ s \in [-r,0],
		\end{cases}
	\end{equation}
here $\hat{\alpha}=(0.71,0.61),$ the delay $\tau(t)=\displaystyle\frac{2+\sin t}{3}$ for $t\geq 0$, $r=1$, and
\begin{align*}
     f(w_1,w_2)=\left(\begin{array}{cc}
	-4{w_1}+3w_2 \\ w_1-3{w_2} \end{array}\right),\ g(w_1,w_2)=\left(\begin{array}{cc}{w_1^2}+3\sqrt{w_1^3w_2} \\ {w_1w_2}+2{w_2^2} \end{array}\right).
\end{align*}
It is obvious that $f(\cdot)$ is continuously differentiable on $\R^2\backslash \{0\}$ and
\[
		Df(w_1,w_2)=\begin{pmatrix}-4&3\\ 1&-2 \end{pmatrix}
\]
is a Metzler matrix. Hence, this function is cooperative on $\R^2_{\ge0}$. In addition, $g(\cdot)$ is continuously differentiable on $\R^2\backslash \{0\}$ and is order-preserving on $\R^2_{\ge0}$. On the other hand, $f(\cdot)$ is homogeneous of degree $p=1$ and $g(\cdot)$ is homogeneous of degree $q=2$. These observations show that the assumptions {\textup{\bf{(H1)}}} and {\textup{\bf{(H2)}}} are verified. Now, choosing $v=(0.3,0.2)^{\rm T}$, then the assumption {\textup{\bf{(S)}}} holds because
\[f(v)+g(v)= \left(\begin{array}{cc}
	-0.29 \\ -0.16 \end{array}\right)\prec 0.\]
So, according to Theorem \ref{LMLS}, the trivial solution of \eqref{Eqmain-2} is locally Mittag-Leffler stable. Take the initial condition $\varphi(s)\thickapprox\left(\begin{array}{cc}
	0.2 \\ 0.15 \end{array}\right)$ on the interval $[-1,0]$, we see that $\|\varphi\|_v<1$.  The asymptotic behavior of the solution $\Phi(\cdot,\varphi)$ is depicted in Figure 1.  
 
 When the degree of homogeneity of function $g$ is larger than the one of function $f$, in general, the trivial solution of system \eqref{Eqmain-2} is only locally Mittag-Leffler stable. Indeed, take $\varphi=(1.2,0.4)^{\rm T}$ (it is easy to check that $\|\varphi\|_v>1$), by a numerical simulation, we can see that the solution $\Phi(\cdot,\varphi)$ of the system does not converge to the origin (see Figure 2). 
\begin{figure}
		\begin{center}
			\includegraphics[scale=.8]{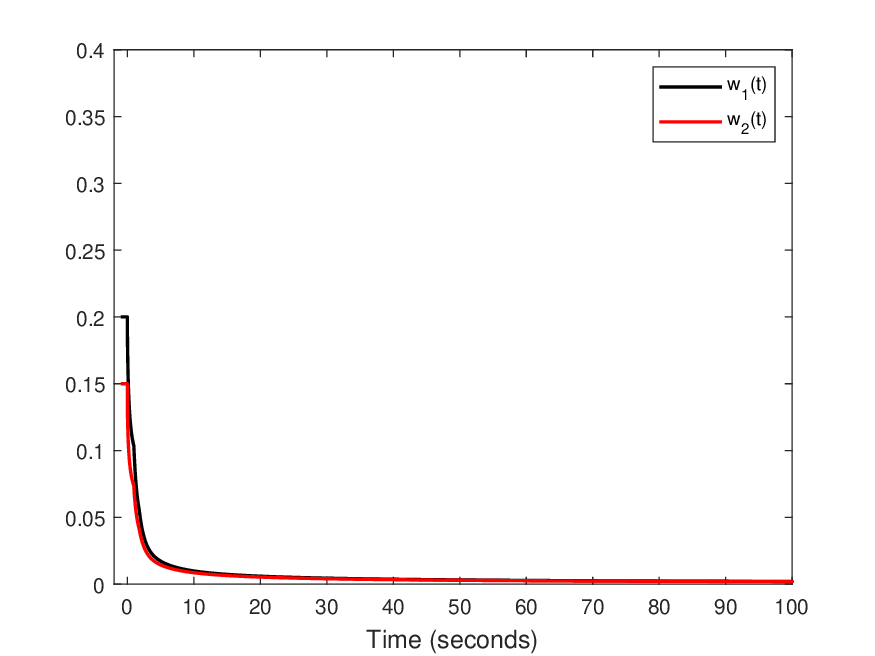}
		\end{center}
		\begin{center}
			\caption{The solution to system \eqref{Eqmain-2} with $\varphi(s)=\left(\begin{array}{cc}
	0.2\\0.15 \end{array}\right)$ on the interval $[-1,0]$.}
		\end{center}
	\end{figure}
 \begin{figure}
		\begin{center}
			\includegraphics[scale=.8]{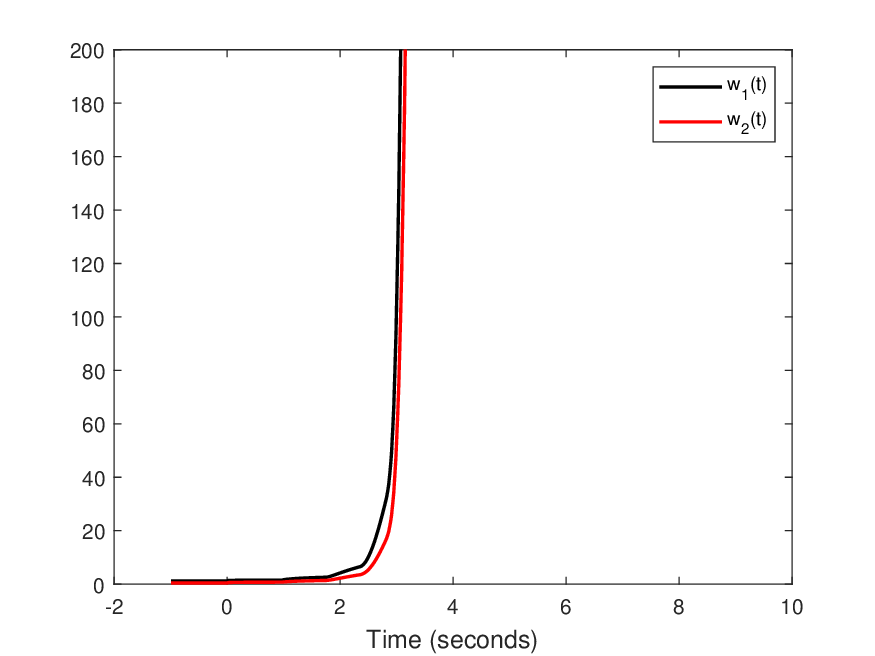}
		\end{center}
		\begin{center}
			\caption{The solution to the system \eqref{Eqmain-2} with $\varphi(s)=\left(\begin{array}{cc}
	1.2 \\ 0.4 \end{array}\right)$ on the interval $[-1,0]$.}
		\end{center}
	\end{figure}
\end{example}
\begin{example}
	Consider the system
	\begin{equation} \label{Eqmain-1}
		\begin{cases}
		^{\!C}D^{\hat{\alpha}}_{0^+}w(t)&=f(w(t))+g(w(t-\tau(t)),\ \forall t> 0,\\
			w(s)&=\varphi(s),\ s \in [-r,0].
		\end{cases}
	\end{equation}
Here, we choose $\hat{\alpha}=(0.95,\ 0.7),$ $\tau(t)=\displaystyle\frac{1}{2}+\displaystyle\frac{1}{2+t^2}$ for $t\geq 0$, $r=1$, and
\begin{align*}
     f(w_1,w_2)=\left(\begin{array}{cc}
	-8{w_1^2}+w_2^2 \\ 2w_1^2-9{w_2}^2 \end{array}\right),\ g(w_1,w_2)=\left(\begin{array}{cc} 3w_1w_2+w_2^2 \\ {(w_1+2w_2)}\sqrt{w_1^2+7w_2^2} \end{array}\right).
\end{align*}
Function $f(\cdot)$ is continuously differentiable on $\R^2\backslash \{0\}$ and
\[
		Df(w_1,w_2)=\begin{pmatrix}-16w_1&2w_2\\ 4w_1&-18w_2 \end{pmatrix}.
\]
Hence, it is cooperative on $\R^2_{\ge0}$. Function $g(\cdot)$ is continuously differentiable on $\R^2\backslash \{0\}$ and is order-preserving on $\R^2_{\ge0}$. Furthermore, $f(\cdot),\ g(\cdot)$ are homogeneous of degree $2$. Thus, the assumptions {\textup{\bf{(H1)}}} and {\textup{\bf{(H2)}}} are satisfied. Choosing $v=(1,\ 1)^{\rm T}$, then 
\[f(v)+g(v)= \left(\begin{array}{cc}
	-3 \\ -1 \end{array}\right)\prec 0.\]
 This implies that the assumption {\textup{\bf{(S)}}} is true. Then, by Theorem \ref{LMLS}, the trivial solution of \eqref{Eqmain-2} is globally Mittag-Leffler stable. Figures 3 (the initial condition $\|\varphi\|_v<1$) and Figure 4 (the initial condition $\|\varphi\|_v>1$) illustrate the fact that every solution of system \eqref{Eqmain-1} is attracted to the origin regardless of whether its initial condition is small or large.
\begin{figure}
		\begin{center}
			\includegraphics[scale=.8]{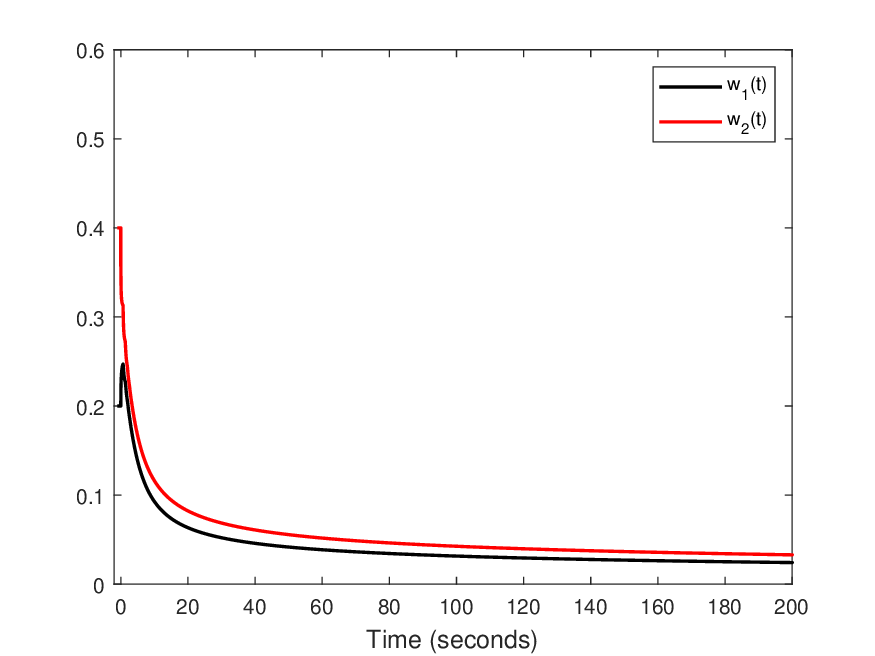}
		\end{center}
		\begin{center}
			\caption{The solution to system \eqref{Eqmain-1} with $\varphi(s)=\left(\begin{array}{cc}
	0.2 \\ 0.4 \end{array}\right)$ on the interval $[-1,0]$.}
		\end{center}
	\end{figure}
\begin{figure}
		\begin{center}
			\includegraphics[scale=.8]{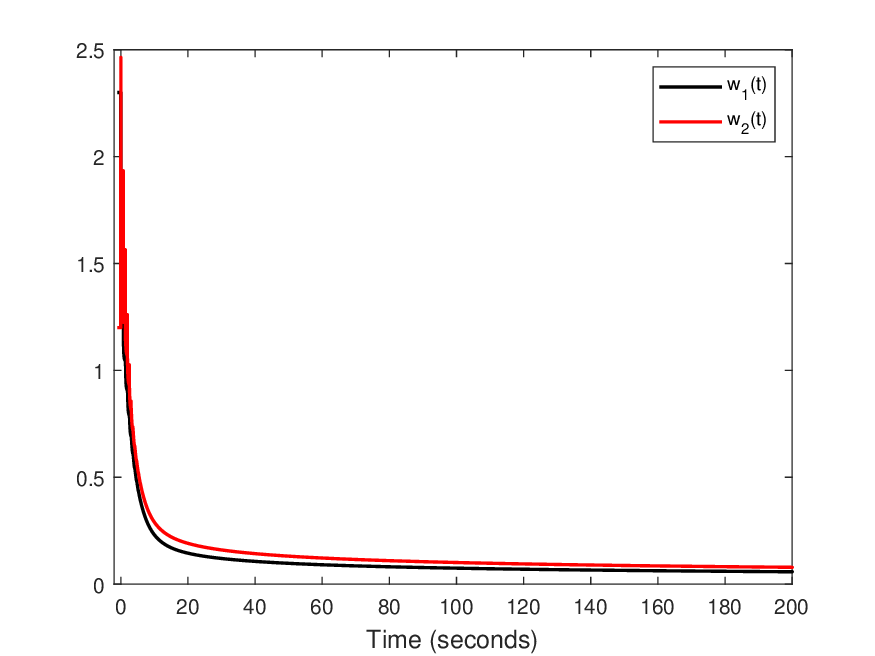}
		\end{center}
		\begin{center}
			\caption{The solution to system \eqref{Eqmain-1} with $\varphi(s)=\left(\begin{array}{cc}
	2.3 \\ 0.2 \end{array}\right)$ on the interval $[-1,0]$.}
		\end{center}
	\end{figure}
\end{example}
\section*{Acknowledgments}
The authors would like to express their gratitude to Professor Jinqiao Duan (Department of Mathematics, Great Bay University, Dongguan, Guangdong, China) for his support and helpful discussions.

\end{document}